\documentclass[pdftex, a4paper, 12pt]{amsart}
\usepackage{geometry}
\usepackage[cp1250]{inputenc}

\usepackage{alltt}
\usepackage{euler}
\usepackage{amsfonts}
\usepackage{amscd}
\usepackage{graphicx}
\usepackage{lpic}
\usepackage{longtable}
\usepackage{url}
\usepackage{import}
\usepackage{amssymb, amsthm, amsmath}
\usepackage[sc]{mathpazo}
\usepackage{enumitem, textcomp, setspace}

\usepackage{tikz}
\usetikzlibrary{intersections}
\usepackage{pgfplots}
\usepgfplotslibrary{external}

\allowdisplaybreaks

\numberwithin{equation}{section}

\theoremstyle{plain}
\newtheorem{theorem}{Theorem}[section]
\newtheorem{proposition}[theorem]{Proposition}
\newtheorem{lemma}[theorem]{Lemma}

\theoremstyle{definition}
\newtheorem{definition}[theorem]{Definition}
\newtheorem{example}[theorem]{Example}

\usepackage{booktabs}

\usepackage[pagebackref=false]{hyperref}

\usepackage[all]{hypcap}

\definecolor{newblue}{rgb}{0.27, 0.32, 0.86}
\definecolor{newred}{rgb}{0.86, 0.32, 0.27}
\hypersetup{
	pagebackref=true,
	colorlinks=true,
	linkcolor=newred,     
	citecolor=newblue,   
	filecolor=magenta,
	urlcolor=newblue
}

\graphicspath{ {figures/} } 

\def\utr{\, \underline{\triangleright}\, }
\def\otr{\, \overline{\triangleright}\, }

\title[Biquandle cocycle condition for invariants of immersed surface-links]{Biquandle cocycle condition for invariants\\ of immersed surface-links in the four-space}

\author{Micha{\l} Jab{\l}onowski}

\address{Institute of Mathematics, Faculty of Mathematics, Physics and Informatics,\newline University of Gda\'nsk, 80-308 Gda\'nsk, Poland}



\email{michal.jablonowski@gmail.com}

\date{\today}

\begin{document}

\maketitle

\begin{abstract}
We consider a biquandle-cohomological framework for invariants of oriented immersed surface-links in the four-space. After reviewing projections and Roseman moves for immersed surfaces, we prove that the move types (a, b, c, e, f, g, h) form a minimal generating set, showing in particular that the singular move (h) is independent of the embedded-case set (a, b, c, e, f, g). We extend biquandle colorings to broken surface diagrams with singular points and establish that coloring sets are in bijection for diagrams related by these moves, yielding a coloring number invariant for immersed surface-links. We introduce singular biquandle 3-cocycles: biquandle 3-cocycles satisfying an additional antisymmetry when the singular relations hold. Using such cocycles, we define a triple-point state-sum with Boltzmann weights and prove its invariance under all generating moves, including (h), thereby obtaining a state-sum invariant for immersed surface-links. The theory is illustrated on the Fenn-Rolfsen link example, where a computation yields a non-trivial integer value, demonstrating the nontriviality of the invariant in the immersed setting. These results unify and extend biquandle cocycle invariants from embedded to immersed surface-links.
\end{abstract}

\renewcommand*{\arraystretch}{1.4}

\section{Introduction}\label{sec1}

We work in the smooth ($C^{\infty}$) category. All manifolds will be assumed to be compact. Let $X$ and $Y$ be smooth manifolds. Let $f:X^n\to Y^{m}$ be a smooth map. It is called an \emph{immersion} if at each point $x\in X$ the induced differential is a monomorphism. In the case $f:X^2\to \mathbb{S}^{4}$, we have generically the finite point self-transverse intersections. An immersion (or its image when no confusion arises) of a closed (i.e., compact, without boundary) surface $F$ into $\mathbb{S}^4=\mathbb{R}^4\cup\{\infty\}$ is called an \emph{immersed surface-link} (or \emph{immersed surface-knot} if it is connected). 
\par 
Two immersed surface-links are \emph{equivalent} if there exists an orientation-preserving homeomorphism of the four-space $\mathbb{R}^4$ to itself (or equivalently an auto-homeomorphism of the four-sphere $\mathbb{S}^4$), mapping one of those surfaces onto the other. Fix an immersed surface-link $F$ in a manifold $\mathbb{S}^4$. For an open tubular neighborhood, denoted $N(F)$, the \emph{exterior} of $F$ is $E(F):=\mathbb{S}^4\backslash N(F)$.
\par 
This paper is organized as follows. In Section \ref{sec2}, we review the moves on projections of immersed surfaces into $\mathbb{R}^3$ and prove that the generating set of the moves is a minimal generating set. In Section \ref{sec3}, we review the biquandle structure and its applications to the broken surface method of examining embedded surface-links, and we extend the theory to the singular case. In Section \ref{sec4}, we derive conditions for a biquandle $3$-cocycle to later define the state-sum invariant of immersed surface-links in $\mathbb{R}^3$. Finally, we give an example of an immersed surface-link and a singular $3$-cocycle, such that the computed singular cocycle invariant takes a nontrivial integer value.

\section{Moves on projections of immersed surfaces}\label{sec2}

Let us fix a generic projection $\pi:\mathbb{R}^4\to\mathbb{R}^3$ by \emph{$\pi((x, y, z, t))=(x, y, z)$}. A \emph{double decker set} $\Gamma$ of a surface $F$ is the closure in $F$ of a set $\{p\in F:\#(\pi^{-1}(\pi(p))\cap F)>1\}$, and the \emph{double point set} is the image $\pi(\Gamma)$. Without loss of generality, we may assume that the image of projection $\pi$ is in \emph{general position}, i.e., a neighborhood of a surface projection is locally homeomorphic to one of the cases shown in Figure \ref{Singular}.

\begin{figure}[h!t]
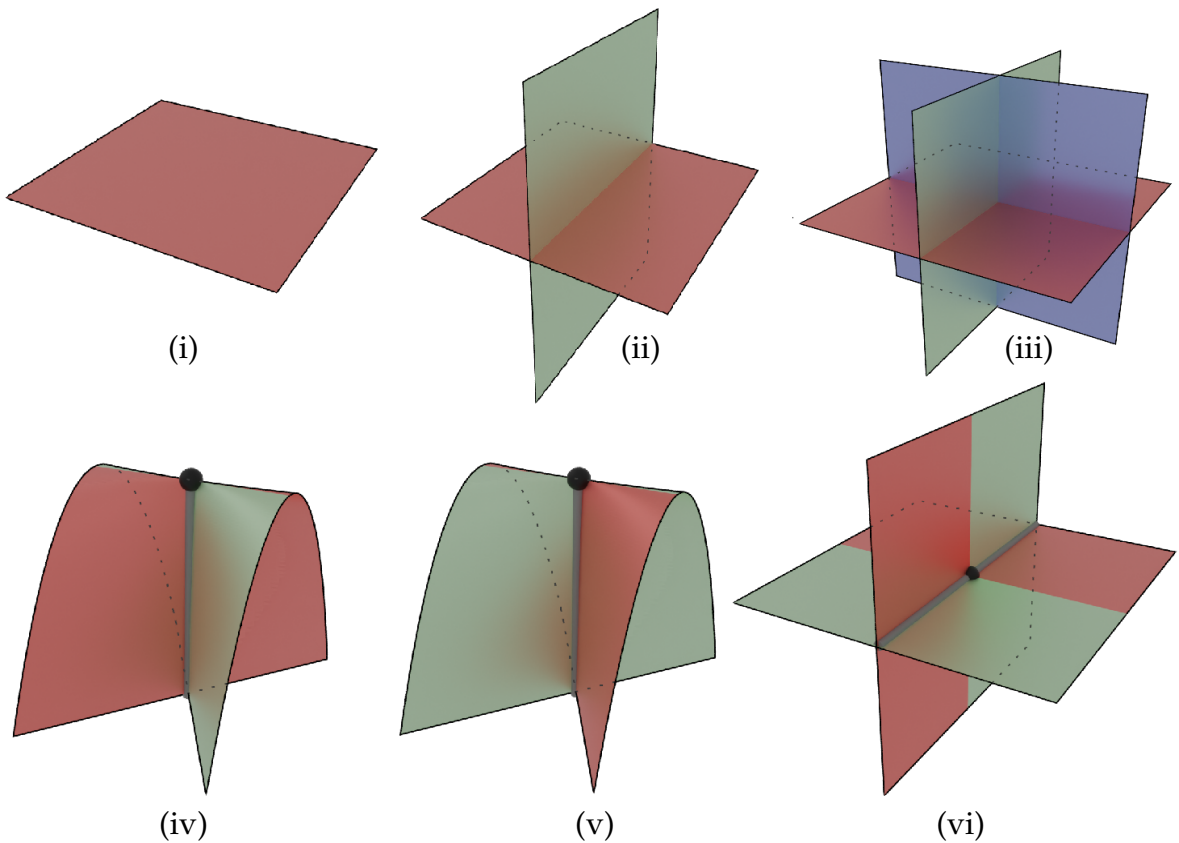

	\begin{center}
		\begin{lpic}[b(0.5cm)]{Singular.png(15.5cm)}
			\lbl[t]{80,210;(i)}
			\lbl[t]{280,210; (ii)}
			\lbl[t]{450,210; (iii)}
			\lbl[t]{80,0;(iv)}
			\lbl[t]{260,0; (v)}
			\lbl[t]{420,0; (vi)}
		\end{lpic}
		\caption{Possibilities of a neighborhood of a surface projection.\label{Singular}}
	\end{center}
\end{figure}

\par 
Points in the projection corresponding to cases (i)--(vi) are called a \emph{regular point} (i), a \emph{double point} (ii), a \emph{triple point} (iii), a \emph{negative branch point} (iv), a \emph{positive branch point} (v), and a \emph{singular point} (vi), respectively. In Figure \ref{Singular}, there are black dots that indicate the position of the branch or the singular points, and around the dots, there are gray segments that indicate where the double point set is. 

A \emph{diagram} $D_K$ of a knotted surface $K$ is the image $\pi(K)$ with additional information in the small neighborhood of the self-crossing, about which piece of surface was higher before being projected.
\par 
A common way to indicate over/under information on a diagram is to make the under-sheets be broken along the double point set curves, and such a broken surface is called a \emph{broken surface diagram} of a surface-link. Instead, we will use color information, such that the red part of the diagram lies farther, with respect to the projection direction, than the green, the green farther than the blue, and the blue farther than the yellow. For later use of algebraic coloring of the diagram, we will use the broken surface diagram terminology like \emph{semi-sheets} (without making actual cuts).

\begin{figure}[h!t]
	\begin{center}
		\begin{lpic}[]{RoseB.png(16.5cm)}
			\lbl[t]{110,370;$(a)\atop \longleftrightarrow$}
			\lbl[t]{360,370;$(b)\atop \longleftrightarrow$}
			\lbl[t]{115,300;$(c)\atop \longleftrightarrow$}
			\lbl[t]{365,300;$(e)\atop \longleftrightarrow$}
			\lbl[t]{110,210;$(f)\atop \longleftrightarrow$}
			\lbl[t]{360,210;$(g)\atop \longleftrightarrow$}
			\lbl[t]{220,95;$(h)\atop \longleftrightarrow$}
		\end{lpic}
		\caption{Moves.\label{RoseB}}
	\end{center}
\end{figure}

We have the Roseman moves (a, b, c, e, f, g) presented in \cite{Ros98} and the move (h) presented in \cite{AMW17}. We present them in Figure \ref{RoseB}. In \cite{Ros98}, Roseman introduced seven moves; the additional move (d), not shown, follows from the other six Roseman moves \cite{Yas05}.
\par 
For simplicity, we present only one valid combination of the height colors for each of those moves. The other valid moves are understood to be the mirror moves (changing the color order to some other admissible one); more details are in \cite{CKS04}.

\begin{theorem}[\cite{Ros98, AMW17}]
	Two diagrams represent equivalent immersed surface-links if and only if one of them may be achieved from the other by a finite sequence of moves (a, b, c, e, f, g, h) allowing an isotopy of the diagram in $\mathbb{R}^3$.
\end{theorem}

Those moves were helpful in establishing invariants of surface-link diagrams (see \cite{Kam17}) and diagrams without branch or triple points with respect to the appropriate equivalences in \cite{Sat01} and \cite{Jab12}, respectively. The relationship among the seven types of Roseman moves were clarified in \cite{Kaw15}, where it is shown that the type of moves (a, b, c, e, f, g) is the minimal generating set for Roseman moves. We extend this to the case of immersed surfaces.

\begin{proposition}
	The set of types of moves (a, b, c, e, f, g, h) is a minimal generating set for any finite combination of types of moves (a, b, c, e, f, g, h).
\end{proposition}

\begin{proof}
	The new set extends the minimal generating set (a, b, c, e, f, g) in the embedded surface case, by exactly one type -- move (h). It is therefore sufficient to prove the independence of (h) type move from any combination of a set of type (a, b, c, e, f, g) moves, i.e., that it cannot be a combination of these moves (and an ambient isotopy in $\mathbb{R}^3$).
	\par 
	Let $\mathcal D$ be a broken surface diagram of an oriented immersed surface-link $\mathcal L$. Let $\tau$ be a triple point of $\mathcal D$ intersecting three sheets that have relative positions top, middle, and bottom with respect to the projection direction of $p:\mathbb R^4\rightarrow \mathbb R^3.$ The \emph{sign} of the triple point $\tau$ is defined to be {\it positive} if the normal vectors of top, middle, bottom sheets in this order match the fixed orientation of the $3$-space $\mathbb R^3$. Otherwise, the \emph{sign} of $\tau$ is defined to be \emph{negative}. We use the right-handed rule convention for the orientation of $\mathbb{R}^3$.
	The generic surface in $\mathbb R^3$ divides into some regions. Then there exists a checkerboard coloring (white and black) of the regions such that the adjacent regions receive the opposite colors \cite{NS98}.
	\par 
	Let our semi-invariant of the diagram be $f_*(\mathcal D) = 2(T_+-T_-)-(W_+-W_-)+(B_+-B_-)$, where $T_+$ (respectively $T_-$) is the number of the positive (respectively negative) triple points; $B_+$ (respectively $B_-$) is the number of the positive (respectively negative) black branch points; $W_+$ (respectively $W_-$) is the number of the positive (respectively negative) white branch points. The neighborhood of a branch point on a
	generic surface in $\mathbb R^3$ looks like the cone on the figure eight. Then the branch point is \emph{black} (respectively \emph{white}) if the regions inside this figure eight is black (respectively white), for details see \cite{Sat00}. In \cite{CS97}, it is shown that $f_*(\mathcal D)=0$ for any broken surface diagram of an embedded surface in the four-space (an elementary proof can be found in \cite{Sat00}). Therefore, it must be invariant with respect to (a, b, c, e, f, g) type moves.
	\par
	Notice that the move (h) changes the sign of the triple point involved, and does not change either signs or colors of the branch points. Therefore, $f_*(\mathcal D)$ always changes by $\pm 4$ after performing an (h) move. To construct now two diagrams of isotopic immersed surface-links that differ by the value of this semi-invariant, take any diagram with at least one immersed point, and perform the move (h), taking the nearest regular sheet as the highest/lowest regular sheet taking part in the move.
\end{proof}

\section{Biquandle structure}\label{sec3}

A generalization of quandles (called biquandles) was introduced in \cite{KR03}. A \emph{biquandle} is an algebraic structure with two binary operations satisfying certain conditions, which can be presented by semi-sheets of surface-link projections as its generators modulo oriented  Roseman moves. In \cite{CES04}, J. S. Carter, M. Elhamdadi, and M. Saito introduced and used cocycles to define invariants via colorings of link diagrams by biquandles and a state-sum formulation.
\par 
In \cite{KKKL18}, S. Kamada, A. Kawauchi, J. Kim, and S. Y. Lee discussed the (co)homology theory of biquandles and developed the biquandle cocycle invariants for oriented surface-links by using broken surface diagrams generalizing quandle cocycle invariants. They showed how to compute the biquandle cocycle invariants from marked graph diagrams.

\begin{definition}
	Let $X$ be a set. A \textit{biquandle structure} on $X$ is a
	pair of maps $\utr,\otr:X\times X\to X$ satisfying:
	\begin{itemize}
		\item[(i)] for all $x\in X$, $x\utr x=x\otr x$,
		\item[(ii)] the maps $\alpha_y,\beta_y:X\to X$ for all $y\in X$ 
		and $S:X\times X\to X\times X$ defined by
		\[\alpha_y(x)=x\otr y,\quad \beta_y(x)=x\utr y\quad \mathrm{and}\quad
		S(x,y)=(y\otr x, x\utr y)\]
		are invertible, and
		\item[(iii)] for all $x,y,z\in X$ we have the \textit{exchange laws}:
		
		\begin{enumerate}
			\item 			$(x\utr y)\utr (z\utr y) =  (x\utr z)\utr(y\otr z),$
			\item 			$(x\otr y)\utr (z\otr y) =  (x\utr z)\otr(y\utr z),$
			\item 			$(x\otr y)\otr (z\otr y) =  (x\otr z)\otr(y\utr z).$
		\end{enumerate}
		
	\end{itemize}
	Axiom (ii) is equivalent to the \emph{adjacent labels rule},
	which says that in the ordered quadruple $(x,y,x\utr y,y\otr x)$, any two 
	neighboring entries (including $(y\otr x, x)$ determine the other two.
	A \emph{biquandle} is a set $X$ with a choice of biquandle structure.
\end{definition}

Let $X$ and $Y$ be biquandles. A function $f:X\rightarrow Y$ is called a {\it biquandle homomorphism} if $f({x}\utr{y})={f(x)}\utr {f(y)}$ and $f({x}\otr {y})={f(x)}\otr {f(y)}$ for any $x,y\in X.$ We denote the set of all biquandle homomorphisms from $X$ to $Y$ by ${\rm Hom}(X,Y)$. A bijective biquandle homomorphism is called a \emph{biquandle isomorphism}. Two biquandles $X$ and $Y$ are said to be \emph{isomorphic} if there is a biquandle isomorphism $f: X \to Y$. 
\par 
From now on, we assume that immersed surface-links are oriented with a fixed orientation. Let $\mathcal B$ be a broken surface diagram of an immersed surface-link $\mathcal L$ in $\mathbb R^4$ and let $S(\mathcal B)$ be the set of the semi-sheets in $\mathcal B$. For a given finite biquandle $X$, we define a {\it biquandle coloring} of $\mathcal B$ by $X$ to be a function $\mathcal C:S(\mathcal B) \rightarrow X$ satisfying the following condition. At a double point curve, two coordinate planes intersect locally, and one is the under-sheet and the other is the over-sheet. The under-sheet (resp. over-sheet) is broken into two semi-sheets, say $u_1$ and $u_2$ (resp. $o_1$ and $o_2$). A normal vector of the under-sheet (resp. over-sheet) points toward one of the components, say $o_2$ (resp. $u_2$). If $\mathcal C(u_1)=a$ and $\mathcal C(o_1)=b,$ then we require that $\mathcal C(u_2)=a\utr{b}$ and $\mathcal C(o_2)=b\otr{a}$, see Figure \ref{30_32} (left). We can see that around the triple point, the conditions are guaranteed by the biquandle axioms, see Figure \ref{30_32} (right). Additionally, around a singular point we require the relations $a\utr{b}=a\otr{b}$ and $b\utr{a}=b\otr{a}$ see Figure \ref{28_29}.

\begin{figure}[h!t]
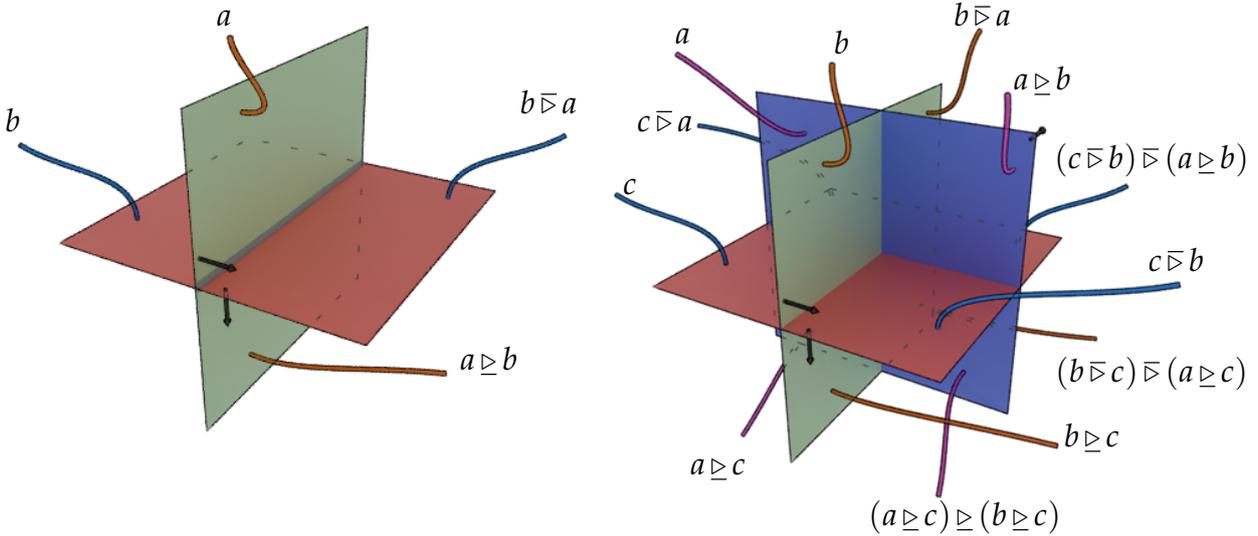

	\begin{center}
		\begin{lpic}[t(0.3cm),l(0.3cm),r(0.8cm),b(0.4cm)]{30_32.png(15.5cm)}
			\lbl[b]{0,110;$b$}
			\lbl[b]{60,140;$a$}
			\lbl[b]{152,115;$b\otr a$}
			\lbl[b]{135,40;$a\utr b$}
			\lbl[b]{190,135;$a$}
			\lbl[b]{235,132;$b$}
			\lbl[b]{175,92;$c$}
			\lbl[b]{185,110;$c\otr a$}
			\lbl[b]{275,140;$b\otr a$}
			\lbl[b]{292,120;$a\utr b$}
			\lbl[b]{200,10;$a\utr c$}
			\lbl[b]{306,18;$b\utr c$}
			\lbl[b]{330,70;$c\otr b$}
			\lbl[b]{323,97;$(c\otr b)\otr (a\utr b)$}
			\lbl[b]{323,36;$(b\otr c)\otr (a\utr c)$}
			\lbl[b]{270,-5;$(a\utr c)\utr (b\utr c)$}
			
		\end{lpic}
		
		\caption{Biquandle conditions around a double point and a triple point.\label{30_32}}
	\end{center}
\end{figure}

\begin{figure}[h!t]
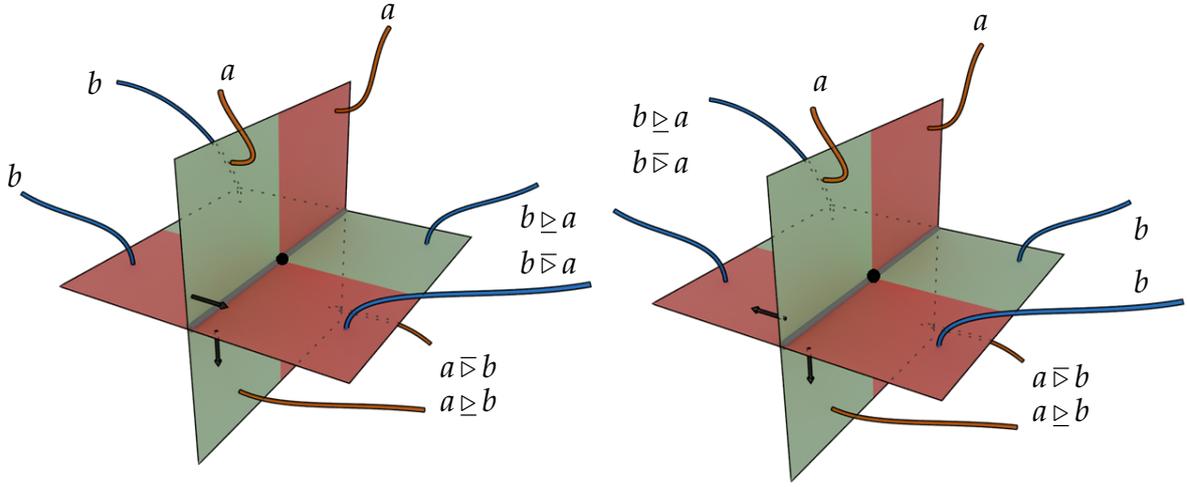

	\begin{center}
		\begin{lpic}[t(0.3cm),l(0.3cm),r(0.1cm),b(0.4cm)]{28_29.png(15.5cm)}
\lbl[b]{0,115;$b$}
\lbl[b]{30,150;$b$}
\lbl[b]{80,155;$a$}
\lbl[b]{140,178;$a$}
\lbl[b]{200,97;$b\utr a$}
\lbl[b]{200,82;$b\otr a$}
\lbl[b]{170,28;$a\utr b$}
\lbl[b]{170,43;$a\otr b$}
\lbl[b]{242,120;$b\otr a$}
\lbl[b]{242,135;$b\utr a$}
\lbl[b]{302,151;$a$}
\lbl[b]{362,174;$a$}
\lbl[b]{422,95;$b$}
\lbl[b]{422,75;$b$}
\lbl[b]{392,24;$a\utr b$}
\lbl[b]{392,39;$a\otr b$}
		\end{lpic}
		\caption{Biquandle conditions around a singular point.\label{28_29}}
	\end{center}
\end{figure}

The biquandle element $\mathcal C(s)$ assigned to a semi-sheet $s$ by a biquandle coloring is called a \emph{color} of $s$. Using biquandle axioms, it is easily checked that the above conditions are compatible at each triple point, branch point, and singular point of $\mathcal B$.
\par 
We denote by ${\rm Col}^B_X(\mathcal B)$ the set of all biquandle colorings of $\mathcal B$ by a biquandle $X$. 

\begin{theorem}[\cite{Car09, KKKL18}]\label{one-one}
	Let $\mathcal L$ be a surface-link and let $\mathcal B$ and $\mathcal B'$ be two broken surface diagrams of $\mathcal L$. Then for any finite biquandle $X$, there is a one-to-one correspondence between ${\rm Col}^{\rm B}_{X}(\mathcal B)$ and ${\rm Col}^{\rm B}_{X}(\mathcal B')$. 
\end{theorem}

We extend Theorem \ref{one-one} to the case of an immersed surface-link.

\begin{proposition}\label{one-one2}
	Let $\mathcal L$ be an immersed surface-link and let $\mathcal B$ and $\mathcal B'$ be two broken surface diagrams of $\mathcal L$. Then for any finite biquandle $X$, there is a one-to-one correspondence between ${\rm Col}^{\rm B}_{X}(\mathcal B)$ and ${\rm Col}^{\rm B}_{X}(\mathcal B')$.
\end{proposition}

\begin{proof}
	We are left to verify the claim of Proposition \ref{one-one2} for the case that $\mathcal B$ and $\mathcal B'$ differ by move (h) and all its possible variations. The uniqueness of the coloring is presented in Figures \ref{h1}--\ref{h2}. We use the biquandle axioms and the singular relations $a\utr{b}=a\otr{b}$ and $b\utr{a}=b\otr{a}$.
\end{proof}

We call the cardinal number $\#{\rm Col}^B_X(\mathcal L)$ the \emph{biquandle $X$-coloring number} of $\mathcal L$. This invariant can naturally be computed from the singular marked diagram \cite{Jab25}.

\begin{figure}[h!t]
	\begin{center}
		\begin{lpic}[b(0.9cm),t(0.2cm),l(0.7cm)]{h1b.png(15.5cm)}
			\lbl[b]{210,-50;\framebox[2cm]{$-\theta(a, b, c)$}}
			\lbl[b]{670,-50;\framebox[2cm]{$+\theta(b, a, c)$}}
			\lbl[b]{122, 365;$a$}
			\lbl[b]{222, 360;$a\utr c$}
			\lbl[b]{320, 320;$c$}
			\lbl[b]{390, 245;$b\utr c$}
			\lbl[b]{380, 170;$b$}
			\lbl[b]{300, 10;$c\otr b$}
			\lbl[b]{60, 20;$a\otr b$}
			\lbl[b]{15, 240;$b\utr a$}
			\lbl[b]{80, 350;$c\otr a$}
			\lbl[b]{35, 284;$(b\utr a)\utr(c\otr a)$}
			\lbl[b]{170, -7;$(c\otr b)\otr(a\otr b)$}
			\lbl[b]{35, 71;$(a\otr b)\utr(c\otr b)$}
			
			\lbl[b]{592, 365;$a$}
			\lbl[b]{692, 360;$a\utr c$}
			\lbl[b]{790, 320;$c$}
			\lbl[b]{860, 245;$b\utr c$}
			\lbl[b]{850, 170;$b$}
			\lbl[b]{770, 10;$c\otr b$}
			\lbl[b]{530, 20;$a\otr b$}
			\lbl[b]{485, 240;$b\utr a$}
			\lbl[b]{550, 350;$c\otr a$}
			\lbl[b]{480, 282;$(b\utr a)\utr(c\otr a)$}
			\lbl[b]{640, -7;$(c\otr b)\otr(a\otr b)$}
			\lbl[b]{505, 71;$(a\otr b)\utr(c\otr b)$}
			
		\end{lpic}
		\caption{Case h1.\label{h1}}
	\end{center}
\end{figure}

\section{Biquandle cocycle condition}\label{sec4}

We assume familiarity with rack and quandle homology theories (see \cite{CKS04, Kam17}). 

\begin{lemma}\cite{KKKL18}
	A homomorphism $\theta: C^{\rm BR}_3(X) \to A$ is a 3-cocycle of the biquandle cochain complex $C^\ast_{\rm BQ}(X;A)$ if and only if $\theta$ satisfies the following two conditions:
	\begin{itemize}
		\item [(i)] $\theta(a,a,b)=0$ and $\theta(a,b,b)=0$ for all $a, b \in X$.
		\item [(ii)] $\theta(b,c,d)+\theta(a,b,d)+\theta(a\utr{b},c\otr{b},d\otr{b})+\theta(a\utr{d},b\utr{d},c\utr{d})$
		$=\theta(a,c,d)+\theta(a,b,c)+\theta(b\otr{a},c\otr{a},d\otr{a})+\theta(a\utr{c},b\utr{c},d\otr{c})$ for all $a, b, c, d \in X$.
	\end{itemize}
\end{lemma}

\begin{definition}
	We say that $\theta$ is a \emph{singular $3$-cocycle} if it is a $3$-cocycle (i.e., fulfills biquandle $3$-cocycle conditions (i)--(ii)) with the following additional condition.
	
	$(iii)\;$ $\theta(b,a,c)=-\theta(a,b,c)$ and $\theta(c,b,a)=-\theta(c,a,b)$ for all $a, b, c \in X$ such that $b\otr a=b \utr a$ and $a\otr b=a \utr b$.
\end{definition}

Recall the definition of the state-sum invariant for
a surface-link $\mathcal L$ embedded in $\mathbb R^4$, following \cite{KKKL18}. Fix a finite biquandle $X$, an abelian group $A$ written multiplicatively, and a biquandle $3$-cocycle $\theta$. Let $R$ be the source region of a triple point $\tau$ in $\mathcal B$, that is, the octant from which all normal vectors of the three sheets point outwards. For a given biquandle $X$-coloring $\mathcal C$ of $\mathcal B$, let $a,b$, and $c$ be the colors of the bottom, middle, and top sheets, respectively, that bound the source region $R$. Set $\epsilon(\tau)=1$ or $-1$ according as $\tau$ is positive or negative, respectively. The \emph{Boltzmann weight} $W^B_\theta(\tau,\mathcal C)$ at $\tau$ with respect to $\mathcal C$ is defined by $$W^B_\theta(\tau,\mathcal C)=\theta(a,b,c)^{\epsilon(\tau)}\in A.$$ For example, see Figure \ref{28_29} (right), where $\epsilon(\tau)=1$, so the Boltzmann weight of that crossing is $\theta(a,b,c)$.

\begin{figure}[h!t]
	\begin{center}
		\begin{lpic}[b(0.9cm),t(0.2cm),l(0.7cm)]{h2b.png(15.5cm)}
	\lbl[b]{210,-50;\framebox[2cm]{$-\theta(c, a, b)$}}
	\lbl[b]{670,-50;\framebox[2cm]{$+\theta(c, b, a)$}}
	\lbl[b]{122, 365;$a$}
	\lbl[b]{222, 360;$a\otr c$}
	\lbl[b]{320, 320;$c$}
	\lbl[b]{390, 245;$b\otr c$}
	\lbl[b]{380, 170;$b$}
	\lbl[b]{300, 10;$c\utr b$}
	\lbl[b]{60, 20;$a\otr b$}
	\lbl[b]{15, 240;$b\utr a$}
	\lbl[b]{80, 350;$c\utr a$}
	\lbl[b]{35, 284;$(b\utr a)\otr(c\utr a)$}
	\lbl[b]{170, -7;$(c\utr b)\utr(a\otr b)$}
	\lbl[b]{35, 71;$(a\otr b)\otr(c\utr b)$}
	
	\lbl[b]{592, 365;$a$}
	\lbl[b]{692, 360;$a\otr c$}
	\lbl[b]{790, 320;$c$}
	\lbl[b]{860, 245;$b\otr c$}
	\lbl[b]{850, 170;$b$}
	\lbl[b]{770, 10;$c\utr b$}
	\lbl[b]{530, 20;$a\otr b$}
	\lbl[b]{485, 240;$b\utr a$}
	\lbl[b]{550, 350;$c\utr a$}
	\lbl[b]{480, 282;$(b\utr a)\otr(c\utr a)$}
	\lbl[b]{640, -7;$(c\utr b)\utr(a\otr b)$}
	\lbl[b]{505, 71;$(a\otr b)\otr(c\utr b)$}
	
\end{lpic}
		\caption{Case h2.\label{h2}}
	\end{center}
\end{figure}

\begin{definition}
	Let $\mathcal B$ be a broken surface diagram of a surface-link and let $\theta$ be a biquandle $3$-cocycle. The \emph{state-sum} or \emph{partition function} of $\mathcal B$ (associated with $\theta$) is defined to be the sum 
	$$\Phi^B_\theta(\mathcal B; A)=\sum_{\mathcal C \in {\rm Col}^B_X(\mathcal B)}\prod_{\tau\in T(\mathcal B)}W^B_\theta(\tau,\mathcal C),$$ 
	where $T(\mathcal B)$ denotes the set of all triple points of $\mathcal B$ and the state-sum $\Phi^B_\theta(\mathcal B; A)$ is an element of the group ring $\mathbb Z[A].$ 
\end{definition}

\begin{theorem}[\cite{KKKL18}]
	Let $\mathcal L$ be a surface-link and let $\mathcal B$ be a broken surface diagram of $\mathcal L$. For a given biquandle $3$-cocycle $\theta$, the state-sum $\Phi^B_\theta(\mathcal B; A)$ is an invariant of $\mathcal L$. 
	(It is denoted by $\Phi^B_\theta(\mathcal L; A)$.)  
\end{theorem}

We can extend the invariant to the immersed case.

\begin{theorem}
	 The state-sum $\Phi^B_\theta(\mathcal L; A)$, for a singular biquandle $3$-cocycle $\theta$, computed on any broken (immersed) surface diagram for an immersed surface-link $\mathcal L$ is an invariant of $\mathcal L$.
\end{theorem}

\begin{proof}
	It suffices to check invariance under move (h). As shown in Figures  \ref{h1}--\ref{h2} (with co-orientation of the sheets indicated by a black segment), the Boltzmann weight at the affected triple point changes sign, and the two arguments corresponding to the sheets forming the singular point are exchanged. By condition (iii) of a singular $3$-cocycle (written additively; multiplicatively in the state-sum), this leaves the product of Boltzmann weights over all triple points unchanged. Hence, the state-sum is invariant.
	
\end{proof}

\begin{figure}[h!t]
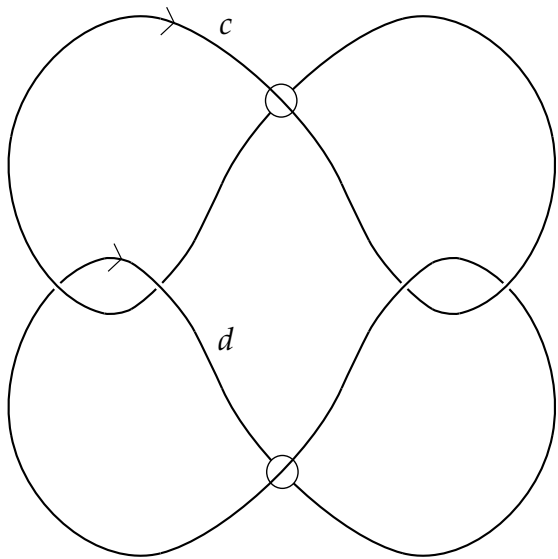

	\begin{center}
		\begin{lpic}[]{fr2(8cm)}
			\lbl[b]{82,79;$d$}
			\lbl[b]{82,184;$c$}
			\lbl[b]{40,115;$g$}
			\lbl[t]{40,87;$h$}
		\end{lpic}
		\caption{The Fenn--Rolfsen link.\label{fr2}}
	\end{center}
\end{figure}

\begin{example}
	
Let us consider the Fenn--Rolfsen link $FR$, shown as the singular marked diagram $D$ in Figure \ref{fr2}. This is the same link as considered in \cite{Jab23}, where we defined the fundamental quandle for an immersed surface-link and computed the coloring invariant for $FR$. We want to compute the state-sum invariant $\Phi(FR)$ that we defined earlier in this paper. 

We give two calculations of the state-sum invariant, first with a quandle biquandle, and then for a non-quandle biquandle.

\end{example}

Each quandle $(Q, *)$ is a special case of a biquandle $(Q, \utr, \otr)$ with the operations $\utr = *$ and $\otr$ is a trivial operation (projection on the first factor).
\par 
We can give the presentation of the fundamental quandle, given the generators ($c, d$) as in Figure \ref{fr2}, $\mathcal Q(FR)=\left<c, d\;|\; c*d=c*(c*d), d*c=d*(c*d), c*(c*d)=c, d*(d*c)=c \right>$.
\par 
We can now use the same method as in \cite[section 7]{KKKL18} and compute the quandle cocycle invariant from both resolutions $L_-(D)$ and $L_{B+}(D)$ of the singular marked graph diagram. That is because a small neighborhood around each triple point need not contain any singular point; therefore, we need only to track the Reidemeister moves of type III in the movie for the resolutions that resolve to trivial unlinks without any crossings.
\par
We present the movies for each resolution in Figures \ref{fr2_Lp}--\ref{fr2_Lm}, which  are based on \cite[p.111-112]{CKS04}. We notice that the unlinking movies of each resolution contain four Reidemeister moves of type III (contributing to our state-sum) and reducing Reidemeister moves of type I (which only change the number of branch points), and some Reidemeister II moves. Contributions of each triple point, marked as $T$ in a triangle involved in the corresponding Reidemeister III move, are shown below the image (the label $\pm(x, y, z)$ corresponds to $\theta(x, y, z)^{\pm 1}$ in the multiplicative convention).

\begin{figure}[h!t]
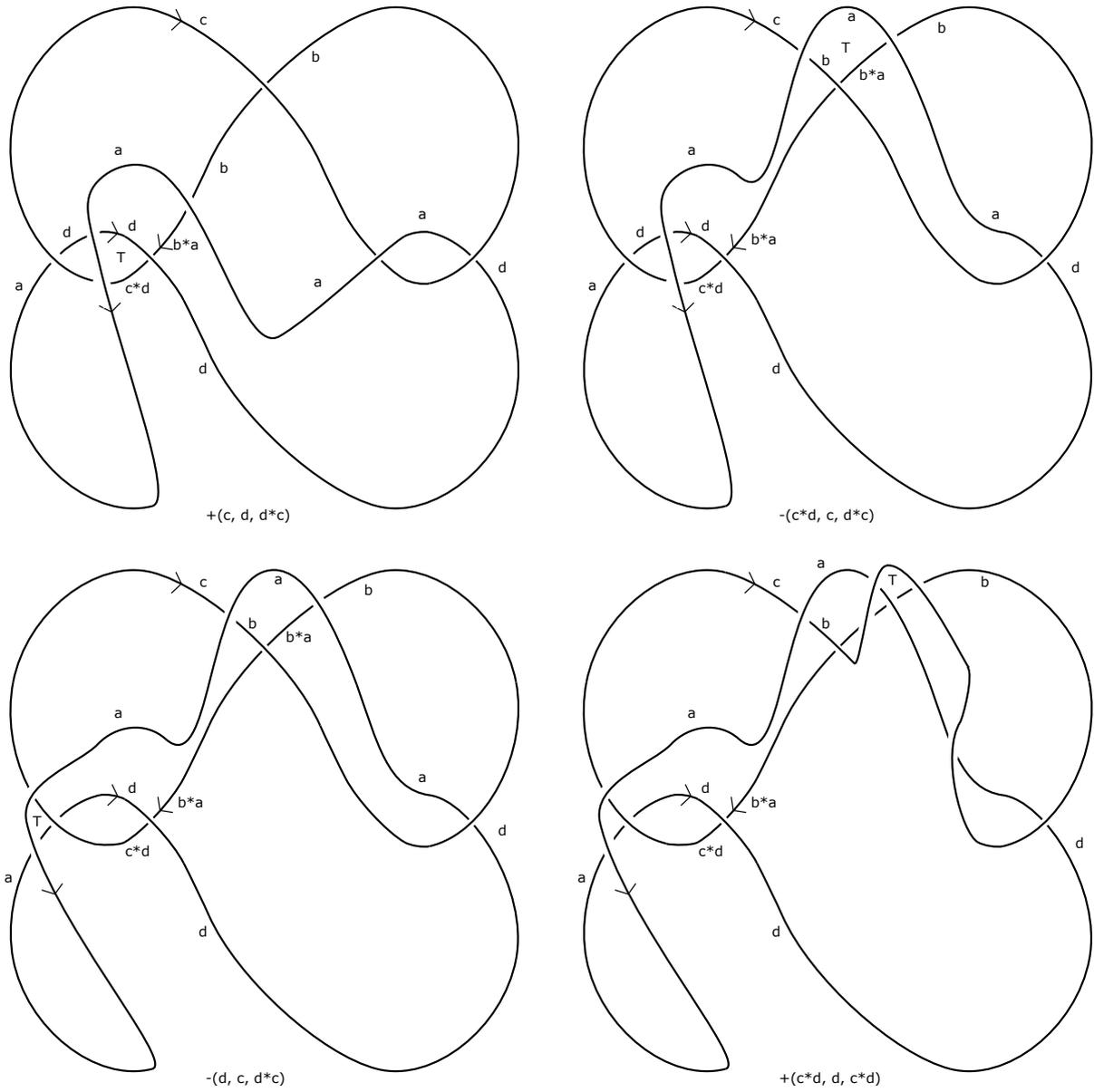

	\begin{center}
		\begin{lpic}[]{fr2_Lp_1(8.2cm)}
		\end{lpic}
		\begin{lpic}[]{fr2_Lp_2(8.2cm)}
		\end{lpic}
		\begin{lpic}[]{fr2_Lp_3(8.2cm)}
		\end{lpic}
		\begin{lpic}[]{fr2_Lp_4(8.2cm)}
		\end{lpic}
		\caption{$L_{B+}$ resolution of $D$.\label{fr2_Lp}}
	\end{center}
\end{figure}

\begin{figure}[h!t]
	\begin{center}
		\begin{lpic}[]{fr2_Lm_1(8.2cm)}
		\end{lpic}
		\begin{lpic}[]{fr2_Lm_2(8.2cm)}
		\end{lpic}
		\begin{lpic}[]{fr2_Lm_3(8.2cm)}
		\end{lpic}
		\begin{lpic}[]{fr2_Lm_4(8.2cm)}
		\end{lpic}
		\caption{$L_{-}$ resolution of $D$.\label{fr2_Lm}}
	\end{center}
\end{figure}
 
Calculations of the state-sum $\Phi$ for a one general quandle coloring $\mathcal C$ generated by elements $c, d$ give the summand (where $\overline{*}$ is the left-inverse of $*$):
\ \\

$$\displaystyle \Phi^Q_{\theta|\mathcal C}(FR):=\theta(c, d, d*c)\cdot\theta(c*d, c, d*c)^{-1}\cdot\theta(d, c, d*c)^{-1}\cdot\theta(c*d, d, c*d)\cdot$$

$$\cdot\theta(c\overline{*}d, d*c, d)^{-1}\cdot\theta(c\overline{*}d, c, d)\cdot\theta(d, c, d)\cdot\theta(c\overline{*}d, d, c)^{-1}.$$
\ \\

Let a quandle on $Q=\mathbb{Z}_9$ be the quandle with the operation given by the following matrix. 
\[
\renewcommand{\arraystretch}{1.04}
\begin{array}{r|rrrrrrrrr}
	* & 0 & 1 & 2 & 3 & 4 & 5 & 6 & 7 & 8\\\hline
	0 & 0 & 6 & 3 & 0 & 6 & 3 & 0 & 6 & 3\\
	1 & 4 & 1 & 7 & 4 & 1 & 7 & 4 & 1 & 7\\
	2 & 8 & 5 & 2 & 8 & 5 & 2 & 8 & 5 & 2\\
	3 & 3 & 0 & 6 & 3 & 0 & 6 & 3 & 0 & 6\\
	4 & 7 & 4 & 1 & 7 & 4 & 1 & 7 & 4 & 1\\
	5 & 2 & 8 & 5 & 2 & 8 & 5 & 2 & 8 & 5\\
	6 & 6 & 3 & 0 & 6 & 3 & 0 & 6 & 3 & 0\\
	7 & 1 & 7 & 4 & 1 & 7 & 4 & 1 & 7 & 4\\
	8 & 5 & 2 & 8 & 5 & 2 & 8 & 5 & 2 & 8
\end{array}
\]

\[
\begin{aligned}
	\theta_S={}&
	\chi_{(0,2,1)}
	\chi_{(0,2,4)}
	\chi_{(0,2,7)}
	\chi_{(0,5,1)}
	\chi_{(0,5,4)}
	\chi_{(0,5,7)}
	\chi_{(0,8,1)}
	\chi_{(0,8,4)}
	\chi_{(0,8,7)}
	\\
	&\cdot
	\chi_{(0,1,2)}^{-1}
	\chi_{(0,1,5)}^{-1}
	\chi_{(0,1,8)}^{-1}
	\chi_{(0,4,2)}^{-1}
	\chi_{(0,4,5)}^{-1}
	\chi_{(0,4,8)}^{-1}
	\chi_{(0,7,2)}^{-1}
	\chi_{(0,7,5)}^{-1}
	\chi_{(0,7,8)}^{-1},
\end{aligned}
\]where $\chi_{(a,b,c)}{(x,y,z)}$ is defined to be $u$ if ${(x,y,z)}={(a,b,c)}$ and $1$ otherwise.\\

We verified computationally that this function satisfies the normalized
biquandle \(3\)-cocycle equations together with the two additional
singular antisymmetry conditions. Hence, \(\theta\) is a singular
biquandle \(3\)-cocycle with coefficients in
$\mathbb{Z}_{3}=\langle u\mid u^{3}=1\rangle .$

\par 
Therefore,
$$\Phi^Q_{\theta_S}(FR)=\sum_{\mathcal C\in Col_Q(FR)}\Phi^Q_{\theta_S|\mathcal C}(FR)=9.$$

{\bf An example involving a non-quandle biquandle.}

We can give the presentation of the fundamental biquandle, given the generators ($g, h$) as in Figure \ref{fr2}, \\

$\mathcal BQ(FR)=\left<g, h\;|\;(h\utr g)\otr(h\otr g)=(h\utr g)\utr(h\otr g),
(h\otr g)\otr(h\utr g)=(h\otr g)\utr(h\utr g),\right.$\\$\left.
(g\utr h)\otr(g\otr h)=(g\utr h)\utr(g\otr h),
(g\otr h)\otr(g\utr h)=(g\otr h)\utr(g\utr h),\right.$\\$\left.
((g\utr h)\utr(g\otr h))\otr((h\otr g)\utr(h\utr g))=((g\otr h)\otr(g\utr h))\utr((h\utr g)\utr(h\otr g)),\right.$\\$\left.
((h\otr g)\otr(h\utr g))\utr((g\utr h)\otr(g\otr h))=((h\utr g)\utr(h\otr g))\otr((g\otr h)\otr(g\utr h))\right>.$
\ \\

Define $\otr^{-1}$, and $\utr^{-1}$ as the left-inverses of $\otr$ and $\utr$ respectively; and $S^{-1}$ as the inverse of $S(x,y)=(y\otr x, x\utr y).$ Calculations of the state-sum $\Phi$ for a one general biquandle coloring $\mathcal C$ generated by elements $g, h$, where each factor is obtained from each diagram in Figures \ref{fr2_Lp}--\ref{fr2_Lm}, give the following.


Put
\[
\alpha
:=
\bigl((g\utr h)\otr(g\otr h)\bigr)
\otr^{-1}(h\utr g).
\]
Then, for a biquandle coloring \(\mathcal C\) generated by \(g\) and \(h\),
the corresponding summand is
\[
\begin{aligned}
	\Phi^{BQ}_{\theta\mid\mathcal C}(FR)
	:={}&
	\theta\bigl(h,g,(g\utr h)\utr^{-1}h\bigr)
	\\
	&\cdot
	\theta\bigl(
	h\utr g,
	h\otr g,
	\alpha
	\bigr)^{-1}
	\\
	&\cdot
	\theta\bigl(
	g,h,(g\utr h)\otr^{-1}h
	\bigr)^{-1}
	\\
	&\cdot
	\theta(x,y,z)
	\\
	&\cdot
	\theta\bigl(
	((g\utr h)\otr(h\utr g))
	\utr^{-1}(g\otr h),
	g\utr h,
	g\otr h
	\bigr)^{-1}
	\\
	&\cdot
	\theta\bigl(
	(h\otr g)\utr^{-1}g,
	h,
	g
	\bigr)
	\\
	&\cdot
	\theta(t,w,t)
	\\
	&\cdot
	\theta\bigl(
	(h\otr g)\utr^{-1}p,
	p,
	v
	\bigr)^{-1}.
\end{aligned}
\]

Here \(x,y,z,w,t,p,v\) are determined by
\[
\begin{aligned}
	(x,y)
	&=S^{-1}(m_{20},m_{17}),\\
	(x,z)
	&=S^{-1}\bigl(
	m_{19},
	(h\utr g)\otr(h\otr g)
	\bigr),\\
	(y,z)
	&=S^{-1}(m_{14},m_{13}),\\
	(w,t)
	&=S^{-1}\bigl(
	m_{21},
	(h\utr g)\otr(h\otr g)
	\bigr),\\
	(p,v)
	&=S^{-1}(h\otr g,g\utr h),
\end{aligned}
\]
where
\[
\begin{aligned}
	m_{20}
	={}&
	\Bigl(
	\bigl((h\otr g)\otr(h\utr g)\bigr)
	\utr
	\bigl((g\utr h)\utr(g\otr h)\bigr)
	\Bigr)
	\utr^{-1}
	\bigl((h\utr g)\otr(h\otr g)\bigr),
	\\[1ex]
	m_{17}
	={}&
	\Bigl(
	\bigl((h\utr g)\otr(g\utr h)\bigr)
	\utr
	\bigl(\alpha\otr(h\otr g)\bigr)
	\Bigr)
	\utr^{-1}
	\bigl((h\otr g)\utr\alpha\bigr),
	\\[1ex]
	m_{19}
	={}&
	(h\utr g)\otr(h\otr g),
	\\[1ex]
	m_{13}
	={}&
	(h\otr g)\otr\alpha,
	\\[1ex]
	m_{14}
	={}&
	(h\utr g)\utr\alpha,
	\\[1ex]
	m_{21}
	={}&
	(g\otr h)
	\otr
	\Bigl(
	\bigl((g\utr h)\otr(h\utr g)\bigr)
	\otr^{-1}(g\otr h)
	\Bigr).
\end{aligned}
\]

	Consider the biquandle $B=\{0,1,2\}$ with operations $\utr$ and $\otr$ given by the following tables.
	
	\medskip
	
	\noindent
	The $\utr$ operation is
	\[
	\renewcommand{\arraystretch}{1.2}
	\begin{array}{c|ccc}
		\utr & 0 & 1 & 2\\ \hline
		0 & 0 & 1 & 2\\
		1 & 1 & 2 & 0\\
		2 & 2 & 0 & 1
	\end{array}
	\]
	and the $\otr$ operation is
	\[
	\renewcommand{\arraystretch}{1.2}
	\begin{array}{c|ccc}
		\otr & 0 & 1 & 2\\ \hline
		0 & 0 & 0 & 0\\
		1 & 2 & 2 & 2\\
		2 & 1 & 1 & 1
	\end{array}
	\]
	
	\medskip
	
	Let $A=\mathbb Z_3=\langle u \mid u^3=1\rangle$.
	Define a function $\theta_S:B^3\to A$ by
	\[
	\theta_S=
	\chi_{(0,1,2)}\,
	\chi_{(2,0,2)}\,
	\chi_{(1,2,1)}^{-1}\,
	\chi_{(2,0,1)}^{-1},
	\]
	where $\chi_{(a,b,c)}(x,y,z)=u$ if $(x,y,z)=(a,b,c)$ and $1$ otherwise. We verified by computation that $\theta_S$ satisfies conditions (i)--(iii), and hence is a singular $3$-cocycle.
	
	\medskip
	
	With this choice of biquandle and cocycle, the state-sum invariant for the
	Fenn--Rolfsen example evaluates to
	\[
	\Phi^B_{\theta_S}(FR)=3.
	\]
	
\subsection{A non-trivial integer-valued example}

We now give an example in which the state-sum has an integer value, but
nevertheless distinguishes the Fenn--Rolfsen surface-link from a trivial
reference surface-link under the assumption that the coloring number of
the latter is a power of the order of the target biquandle.

Let
$
X=\mathbb{Z}_{6}
$
and define two operations on \(X\) by
\[
x \utr y
= x+4y \pmod 6,
\qquad
x \otr y
= 5x \pmod 6.
\]
This is the Alexander biquandle corresponding to the parameters
$
s=5,
t=1,
$ since\\
$
x \utr y
=tx+(s-t)y
=x+4y \pmod 6,
\qquad
x \otr y
=sx
=5x \pmod 6.
$
Both $s$ and $t$ are units in $\mathbb{Z}_{6}$, and hence these
operations define an Alexander biquandle.

The operation tables are
$
\begin{array}{c|cccccc}
	\utr & 0&1&2&3&4&5\\
	\hline
	0&0&4&2&0&4&2\\
	1&1&5&3&1&5&3\\
	2&2&0&4&2&0&4\\
	3&3&1&5&3&1&5\\
	4&4&2&0&4&2&0\\
	5&5&3&1&5&3&1
\end{array}
\quad
\begin{array}{c|cccccc}
	\otr & 0&1&2&3&4&5\\
	\hline
	0&0&0&0&0&0&0\\
	1&5&5&5&5&5&5\\
	2&4&4&4&4&4&4\\
	3&3&3&3&3&3&3\\
	4&2&2&2&2&2&2\\
	5&1&1&1&1&1&1
\end{array}.
$
\\

This biquandle is not a quandle biquandle. Indeed, its second operation
is not trivial; for example,
$
1 \otr 0=5\neq 1.
$

Substituting these operations into the defining relations of the
fundamental biquandle of the Fenn--Rolfsen surface-link, the coloring
conditions for the two generators $g,h\in\mathbb{Z}_{6}$ reduce to
$
4(g-h)=0\pmod 6.
$
Equivalently,
$
g-h\equiv 0\pmod 3.
$
Consequently, the complete set of colorings is\\
$
{Col}^{B}_{X}(FR)
=\{
(0,0),(0,3),
(1,1),(1,4),
(2,2),(2,5),
(3,0),(3,3),
(4,1),(4,4),
(5,2),(5,5)
\}.
$\\
It follows that
$
\#{Col}^{B}_{X}(FR)=12.
$

Let $A=\mathbb Z_3=\langle u \mid u^3=1\rangle$.
Define a function $\theta_S:X^3\to A$ by

\[
\begin{aligned}
	\theta_S={}&
	\chi_{(0,2,1)}
	\chi_{(0,4,5)}
	\chi_{(2,3,1)}
	\chi_{(2,3,2)}
	\chi_{(2,4,2)}
	\chi_{(4,2,4)}
	\chi_{(4,3,4)}
	\chi_{(4,3,5)}
	\\
	&\cdot
	\chi_{(0,1,2)}^{-1}
	\chi_{(0,5,4)}^{-1}
	\chi_{(2,0,4)}^{-1}
	\chi_{(2,0,5)}^{-1}
	\chi_{(2,1,5)}^{-1}
	\chi_{(4,0,1)}^{-1}
	\chi_{(4,0,2)}^{-1}
	\chi_{(4,5,1)}^{-1},
\end{aligned}
\]where $\chi_{(a,b,c)}{(x,y,z)}$ is defined to be $u$ if ${(x,y,z)}={(a,b,c)}$ and $1$ otherwise. 

We verified by computation that $\theta_S$ satisfies conditions (i)--(iii), and hence is a singular $3$-cocycle.

The corresponding state-sum is
therefore the integer
$
\Phi^{X}_{\theta_S}(FR)
=
12.
$ Suppose that for the chosen trivial reference surface-link \(U\), the splitted unlink of two (singular) components, its
state-sum with respect to the same biquandle is of the form
$\Phi^{X}_{\theta_S}(U)=|X|^{2}=6^{2}.$ Since $12\neq 6^{2}$, we obtain
$\Phi^{X}_{\theta_S}(FR)\neq \Phi^{X}_{\theta_S}(U).$
Thus, under this assumption, the integer-valued state-sum distinguishes
the Fenn--Rolfsen surface-link from the trivial reference surface-link.

We emphasize that this example is non-trivial as a biquandle coloring
invariant, although it is not a non-trivial cocycle enhancement: the
state-sum contains only the identity element of the coefficient group.

The computations were verified using Pyhon.


\begin{thebibliography}{99}

\bibitem[AMW17]{AMW17}
B. Audoux, J.B. Meilhan, and E. Wagner, On codimension two embeddings up to link-homotopy, \emph{Journal of Topology}, 10(4) (2017), 1107--1123.

\bibitem[Car09]{Car09} T. Carrell, The surface biquandle, \emph{Pomona College}, 2009.

\bibitem[CES04]{CES04} J.S. Carter, M. Elhamdadi, and M. Saito, Homology Theory for the Set-Theoretic Yang-Baxter Equation and Knot Invariants from Generalizations of Quandles, \emph{Fund. Math.} 184 (2004), 31--54.

\bibitem[CKS04]{CKS04} J.S. Carter, S. Kamada, and M. Saito, Surfaces in $4$-Space, \emph{Encyclopaedia of Mathematical Sciences}, Vol. 142, Low-Dimensional Topology, III. Springer-Verlag, Berlin Heidelberg New York (2004).

\bibitem[CS97]{CS97} J.S. Carter and M. Saito, Normal Euler classes of knotted surfaces and triple points on projections, \emph{Proceedings of the American Mathematical Society} 125(2) (1997), 617--623.


\bibitem[Jab12]{Jab12} M. Jab\l onowski, Knotted surfaces and equivalencies of their diagrams without triple points, \emph{J. Knot Theory Ramifications} 21 (2012), 1250019.

\bibitem[Jab23]{Jab23} M. Jab\l onowski, Minimal generating sets of moves for surfaces immersed in the four-space, \emph{J. Knot Theory Ramifications} 32 (2023), 2350071.

\bibitem[Jab25]{Jab25} M. Jab\l onowski, On biquandle-based invariant of immersed surface-links, Yoshikawa oriented fifth move, and ribbon 2-knots, preprint arXiv:2505.14724 

\bibitem[Kam17]{Kam17} S. Kamada, Surface-Knots in $4$-Space, \emph{Springer Monographs in Mathematics}, Springer (2017).

\bibitem[KKKL18]{KKKL18} S. Kamada, A. Kawauchi, J. Kim, and S.Y. Lee, Biquandle cohomology and state-sum invariants of links and surface-links. \emph{Journal of Knot Theory and Its Ramifications} 27(11) (2018), 1843016.

\bibitem[KR03]{KR03} L.H. Kauffman and D.E. Radford, Bi-oriented quantum algebras, and generalized Alexander polynomial for virtual links, \emph{Contemp. Math.} 318 (2003), 113--140.

\bibitem[Kaw15]{Kaw15} K. Kawamura, On relationship between seven types of Roseman moves, \emph{Topology and its Applications} 196 (2015), 551--557.

\bibitem[NS98]{NS98} J.J. Nuno-Ballesteros and O. Saeki, On the number of singularities of a generic surface with boundary in a $3$-manifold, \emph{Hokkaido Mathematical Journal} 27(3) (1998), 517-544.

\bibitem[Ros98]{Ros98} D. Roseman, Reidemeister-type moves for surfaces in four-dimensional space, in \emph{Knot Theory}, Banach Center Publications 42 (1998), 347--380.

\bibitem[Sat00]{Sat00} S. Satoh. Lifting a generic surface in $3$-space to an embedded surface in $4$-space, \emph{Topology and its Applications} 106(1) (2000), 103--113.

\bibitem[Sat01]{Sat01} S. Satoh, Double decker sets of generic surfaces in 3-space as homology classes, \emph{Illinois J. Math.} 45 (2001), 823--832.

\bibitem[Yas05]{Yas05} T. Yashiro, A note on Roseman moves, \emph{Kobe J. Math}, 22 (2005), 31--38.

\end{thebibliography}
\end{document}